\documentclass{daj}

\usepackage{amsthm, amsmath, amssymb, amsfonts, url, booktabs, tikz, setspace, fancyhdr, enumerate}
\usepackage[margin = 1in]{geometry}

\usepackage{hyperref}

\definecolor{refkey}{gray}{.75}
\definecolor{labelkey}{gray}{.2}

\usepackage[textsize=scriptsize,colorinlistoftodos]{todonotes}

\newtheorem{theorem}{Theorem}[section]

\newtheorem{lemma}[theorem]{Lemma}
\newtheorem{corollary}[theorem]{Corollary}

\newtheorem*{Sid}{Sidorenko's conjecture}

\theoremstyle{definition}

\theoremstyle{remark}

\newcommand{\overbar}[1]{\mkern 1.5mu\overline{\mkern-1.5mu#1\mkern-1.5mu}\mkern 1.5mu}

\newcommand{\EE}{\mathbb{E}}

\newcommand{\HH}{\mathcal{H}}
\newcommand{\GG}{\mathcal{G}}

\newcommand{\FF}{\mathcal{F}}

\newcommand{\mockalph}[1]{}

\tikzstyle{p}+=[fill=black, circle, minimum width = 1pt, inner sep =
1pt]
\tikzstyle{w}+=[fill=white, draw, circle, minimum width = 1pt, inner sep =
1.5pt]

\dajAUTHORdetails{%
  title = {Sidorenko's Conjecture for Blow-ups}, %% please capitalize all significant words
  author = {David Conlon, Joonkyung Lee},
  plaintextauthor = {David Conlon, Joonkyung Lee},
  plaintexttitle = {Sidorenko's Conjecture for Blow-ups}, %%  title without math or LaTeX
  keywords = {Sidorenko's conjecture, graph homomorphism inequalities.},
}   %%% END \dajAUTHORdetails

\dajEDITORdetails{%
   year={2021},
   number={2},
   received={21 February 2019},   % received date: example: 7 January 2017
   published={30 March 2021},  % published date
   doi={10.19086/da.21472},       % XXX = number of paper, e.g. da006 for paper#6
}   %%% END \dajEDITORdetails

\begin{document}

\begin{frontmatter}[classification=text]
%% EDITOR: this will force the keywords to appear right after the Abstract.
%%   If the abstract is too long and would force the keywords off the
%%   front page, please comment out % [classification=text] above
%%   This way the keywords will be floated on the bottom of the first page
%%   even though the Abstract spills over to the next page.

%%% AUTHOR: Title goes here.  This line is optional.  You must use it
%%   if title has footnote attached or requires nontrivial typesetting,
%%   e.g., inclusion of linebreaks to force nice layout.
\title{Sidorenko's Conjecture for Blow-ups} %% please capitalize all significant words

%%% AUTHOR:
%%% List all authors. If you wish, place grant acknowledgements in \thanks.
%%% In brackets include a short tag for each author.
\author[DC]{David Conlon\thanks{Supported by a Royal Society University Research Fellowship and by ERC Starting Grant 676632.}}
\author[JL]{Joonkyung Lee\thanks{Supported by ERC Consolidator Grant PEPCo 724903 and ERC Starting Grant 676632.}}
%\author[laci]{L\'aszl\'o Lov\'asz\thanks{Supported by...}}
%\author[andy]{Andrew Chi-Chih Yao\thanks{Supported by...}}

%%% AUTHOR: Abstract goes here
\begin{abstract}
A celebrated conjecture of Sidorenko and Erd\H{o}s--Simonovits states that, for all bipartite graphs $H$, quasirandom graphs contain asymptotically the minimum number of copies of $H$ taken over all graphs with the same order and edge density. This conjecture has attracted considerable interest over the last decade and is now known to hold for a broad range of bipartite graphs, with the overall trend saying that a graph satisfies the conjecture if it can be built from simple building blocks such as trees in a certain recursive fashion. 

Our contribution here, which goes beyond this paradigm, is to show that the conjecture holds for any bipartite graph $H$ with bipartition $A \cup B$ where the number of vertices in $B$ of degree $k$ satisfies a certain divisibility condition for each $k$. As a corollary, we have that for every bipartite graph $H$ with bipartition $A \cup B$, there is a positive integer $p$ such that the blow-up $H_A^p$ formed by taking $p$ vertex-disjoint copies of $H$ and gluing all copies of $A$ along corresponding vertices satisfies the conjecture. Another way of viewing this latter result is that for every bipartite $H$ there is a positive integer $p$ such that an $L^p$-version of Sidorenko's conjecture holds for $H$.
\end{abstract}
\end{frontmatter}

%%% AUTHOR: body of paper starts here
\section{Introduction}

One of the central problems in extremal graph theory is to estimate the minimum number of copies of a graph $H$ which can be contained in another graph $G$ of given order and edge density. Even when $H$ is a triangle, this problem is highly non-trivial and was only solved fully by Razborov~\cite{R07, R08} in 2008, who used it as the first test case for his influential flag algebra technique. His result was then extended to $K_4$ by Nikiforov~\cite{N11} and to all $K_r$ by Reiher~\cite{Re12} using further ideas.

Part of the difficulty in proving these results is that the behaviour of the minimum number of copies of $K_r$ as a function of the edge density is surprisingly complicated. On the other hand, when $H$ is a bipartite graph, conjectures of Erd\H{o}s and Simonovits~\cite{ESi83} and Sidorenko~\cite{Sid93} suggest that the minimum should be extremely simple, being asymptotically equal to the number of copies of $H$ in a quasirandom graph of the same density.

This attractive conjecture, usually known as Sidorenko's conjecture, is best stated in terms of homomorphisms. A \emph{homomorphism} from a graph $H$ to a graph $G$ is a mapping $f : V(H) \rightarrow V(G)$ such that $(f(u), f(v))$ is an edge of $G$ whenever $(u, v)$ is an edge of $H$. If $h_H(G)$ is the number of homomorphisms from $H$ to $G$, we write $t_H(G) = h_H(G)/|G|^{|H|}$ for the \emph{homomorphism density}, the probability that a uniform random mapping from $V(H)$ to $V(G)$ is a homomorphism. The conjecture is then as follows.

\begin{Sid}
For any bipartite graph $H$ and any graph $G$,
\[t_H(G) \geq t_{K_2}(G)^{e(H)}.\]
\end{Sid}

Sidorenko~\cite{Sid93} himself showed that the conjecture holds for some simple classes of bipartite graph, namely, complete bipartite graphs, even cycles and trees, and for bipartite graphs with at most four vertices on one side. There the matter stood for some time until work of Hatami~\cite{H10}, connecting it with a question of Lov\'asz~\cite{L08} about which graphs define norms, revived interest in the conjecture. In particular, he showed that hypercubes have a certain weak norming property and, hence, that they satisfy Sidorenko's conjecture.

The first significant breakthrough on the conjecture was made by Conlon, Fox and Sudakov~\cite{CFS10}, who used the dependent random choice technique~\cite{FS11} to show that if $H$ is a bipartite graph with a vertex which is complete to the other side, then Sidorenko's conjecture holds for $H$. As a corollary, they showed that this implies an approximate version of the conjecture. An important further advance was then made by Li and Szegedy~\cite{LSz12}, who initiated the application of entropy methods to the conjecture, at first in the guise of logarithmic convexity inequalities. In particular, they found a remarkably concise proof of the result of Conlon, Fox and Sudakov and extended this result to a more general class, which they referred to as reflection trees.

These ideas were developed further by Kim, Lee and Lee~\cite{KLL14}, who proved the conjecture for what they called tree-arrangeable graphs, and then pushed to their (seemingly) natural conclusion by Conlon, Kim, Lee and Lee~\cite{CKLL18, CKLL182} and, independently, by Szegedy~\cite{Sz15}. These works give broad classes of graphs for which Sidorenko's conjecture holds, though it is somewhat hard to do justice to these classes in this limited space. However, the overall trend is that a graph may be shown to satisfy the conjecture if it can be built from simple building blocks such as trees (or weakly norming graphs~\cite{CL16}) in a certain recursive fashion. The main result of this paper is the following, which we believe moves beyond the confines of this paradigm.

\begin{theorem}\label{thm:main}
Let $H$ be a bipartite graph with bipartition $A \cup B$, $\max_{b \in B} \deg(b) = r$ and, for each $1 \leq k \leq r$, let $d_k$ be the number of vertices with degree $k$ in $B$.
Then, if $\binom{|A|}{r}\binom{r}{k}$ divides $d_k$ for each $1\leq k\leq r$, Sidorenko's conjecture holds for $H$.
\end{theorem}

In the proof, critical use is made of a simple mechanism that we call the H\"older trick, which allows us to convert a graph of the type described in Theorem~\ref{thm:main} into a simpler graph to which we can apply the existing techniques. This trick was first observed in~\cite[Section 3]{CKLL18}, but was not exploited to its full potential. To illustrate this key idea, we will discuss the notorious example $K_{5,5}\setminus C_{10}$ in the next section, showing that its `square' satisfies the conjecture.

When the bipartite graph $H$ is regular on one side of the bipartition, a stronger statement, without any divisibility condition, can be obtained which already implies the aforementioned result about  $K_{5,5}\setminus C_{10}$.
\begin{theorem}\label{thm:reg}
Let $H$ be a bipartite graph with partition $A \cup B$ and $\deg(b) = r$ for all $b \in B$. Then, provided $|B| \geq \binom{|A|}{r}$, Sidorenko's conjecture holds for $H$.
\end{theorem}

Given a bipartite graph $H$ with bipartition $A \cup B$ and a positive integer $p$, its blow-up $H_A^p$, or `$p$-th power', relative to $A$ is defined to be the graph formed by taking $p$ vertex-disjoint copies of $H$ and gluing all copies of $A$ along corresponding vertices. That is, we replace each vertex in $B$ with an independent set of order $p$ and connect every vertex in $A$ that was joined to $b \in B$ to every vertex in the corresponding independent set. Since $\binom{|A|}{r}\binom{r}{k}$ divides $|A|!/(|A|-r)!$ for each $k$, a simple corollary of Theorem~\ref{thm:main}, again generalising the $K_{5,5}\setminus C_{10}$ case, is then as follows.

\begin{corollary} \label{cor:main}
For every bipartite graph $H$ with bipartition $A \cup B$, there is a positive integer $p$ such that $H_A^p$ satisfies Sidorenko's conjecture.
In particular, $p=|A|!$ always suffices.
\end{corollary}

This can be viewed as saying that for any bipartite graph $H$ there is an integer $p$ such that an $L^p$-version of the conjecture holds for $H$. To see this, suppose that $|A| = m$ and identify the set $A$ with $[m] = \{1, 2, \dots, m\}$. Now, writing $x_A = (x_1, \dots, x_m)$, where $x_i \in V(G)$ for all $i = 1, 2, \dots, m$, consider the function $t_H(G; x_A)$ which counts the proportion of mappings $f$ from $V(H)$ to $V(G)$ with $f(i) = x_i$ for all $i = 1, 2, \dots, m$ which are homomorphisms. Sidorenko's conjecture for $H$ is clearly equivalent to the statement that $\mathbb{E}_{x_A} t_H(G; x_A) \geq t_{K_2}(G)^{e(H)}$, whereas Corollary~\ref{cor:main} says that for any $H$ there is a positive integer $p$ such that $\mathbb{E}_{x_A} t_H (G; x_A)^p \geq t_{K_2}(G)^{p \cdot e(H)}$.

Another interesting corollary of Theorem~\ref{thm:main} is as follows.

\begin{corollary} \label{cor:tag}
For any bipartite graph $H$, there is another bipartite graph $H'$ such that Sidorenko's conjecture holds for the disjoint union of $H$ and $H'$.
\end{corollary}

To see this, suppose that $H$ has bipartition $A \cup B$ with $|A|=m$, $\max_{b \in B} \deg(b) = r$ and, for each $1 \leq k \leq r$, $d_k$ is the number of vertices of degree $k$ in $B$.
Let $H'$ be a bipartite graph between $A'$ and $B'$, where $|A'|=r$ and $B'$ has exactly $\lceil d_k/(m+r)!\rceil (m+r)!-d_k$ vertices with degree $k$ for each $1 \leq k \leq r$. It is then easy to check that the conditions of Theorem~\ref{thm:main} are satisfied for $H \cup H'$.

For convenience of notation, we will use the language of graphons throughout the paper. A \emph{graphon} is a symmetric measurable function $W$ from $[0,1]^2$ to $[0,1]$, where symmetric in this context means that $W(x,y) = W(y,x)$ for all $(x,y) \in [0,1]^2$. Very roughly, this may be seen as a continuous analogue of the adjacency matrix of a graph. The \emph{homomorphism density} $t_H(W)$ of a graph $H$ in a graphon $W$ is then given by 
\[t_H(W) = \EE \left[\prod_{ij \in E(H)} W(x_i, x_j)\right] = \int_{[0,1]^{v(H)}} \prod_{ij \in E(H)} W(x_i, x_j) \ d\mu^{v(H)},\]
where $\mu$ is the Lebesgue measure on $[0,1]$. Note that we will typically abbreviate integrals with expectations, as above. In this language, Sidorenko's conjecture for a given $H$ is equivalent to saying that
\[t_H(W) \geq t_{K_2}(W)^{e(H)}\]
for every graphon $W$. It is this statement that we will prove in the cases of interest.

\section{A motivating example} \label{sec:example}

We now take a closer look at the graph $M := K_{5,5} \setminus C_{10}$, the simplest graph for which Sidorenko's conjecture is not known, showing that its `square' does satisfy the conjecture. Since $M$ is vertex-transitive, we don't need to distinguish which part of the bipartition is glued and we can simply write $M^2$ rather than $M_A^2$.

The main result of this short section relates the homomorphism density of $M^2$ to the homomorphism density of another graph. To define this graph, let $\HH$ be a family of subsets of a finite set $A$. The \emph{$(A,\HH)$-incidence graph} is the bipartite graph on $A\cup \HH$ such that $a\in A$ and $F\in\HH$ are adjacent if and only if $a\in F$. For $r\leq m$, the \emph{$(m,r)$-incidence graph} is then the $(A,\HH)$-incidence graph where $A = [m]$ and $\HH=\binom{[m]}{r}$.

\begin{theorem}\label{thm:mobius^2}
Let $F_{5,3}$ be the $(5,3)$-incidence graph. Then, for every graphon $W$, $$t_{M^2}(W)\geq t_{F_{5,3}}(W).$$
\end{theorem}

By a result of the authors~\cite{CL16} (discussed in Appendix~\ref{sec:norm}), $F_{5,3}$ is a weakly norming graph, which in turn implies that it satisfies Sidorenko's conjecture, i.e., for any graphon $W$, $t_{F_{5,3}}(W)\geq t_{K_2}(W)^{30}$. Together with Theorem~\ref{thm:mobius^2}, this implies that $M^2$ also satisfies the conjecture.

For the proof of Theorem~\ref{thm:mobius^2}, we introduce some notation that we will use throughout the paper. Given a vector $(x_1, \dots, x_m)$ and $I = \{i_1, \dots, i_s\} \subseteq [m]$ with $i_1 < \dots < i_s$, let $x_I$ be the vector $(x_{i_1}, \dots, x_{i_s})$. Let $S_m$ be the set of all permutations of $[m]$ and let $\FF$ be a family of functions $\{f_I:I\subseteq [m]\}$, where each $f_I$ is a measurable function from $[0,1]^I$ to $[0,1]$. Then, for each $f_I \in \FF$, let
\[\tilde{f_I}(x_I):=\left(\prod_{\phi\in S_m}f_{\phi(I)}(x_I)\right)^{1/m!}.\]
We will need H\"{o}lder's inequality in the following form:
\begin{align}\label{eq:Holder}
\EE\left[\prod_{I\subseteq [m]}\tilde{f_I}(x_I)\right]\leq \prod_{\phi\in S_m}\EE\left[\prod_{I\subseteq [m]}f_{\phi(I)}(x_I)\right]^{1/m!}.
\end{align}
Finally, suppose $H$ is a bipartite graph with bipartition $A\cup B$. For a subset $F$ of $A$ and a graphon $W$, let $\rho(x_F):=\EE_{y}\prod_{i\in F}W(x_i,y)$. Then $t_H(W)$ can be rephrased as
\[t_H(W)=\EE\left[\prod_{v\in B}\rho(x_{N(v)})\right],\]
where the expectation, both here and in the proof below, is over $x_A$.

\begin{proof}[Proof of Theorem~\ref{thm:mobius^2}]
Note that $M$ is isomorphic to the bipartite graph on $Z \cup Z'$ where $Z$ and $Z'$ are two disjoint copies of the group $\mathbb{Z}_5$
and $i \in Z'$ is adjacent to $i-1$, $i$ and $i+1$ in $Z$.
Let $Z_{i}$ be the subset $\{i-1,i,i+1\}$ of $Z$.
Then
\[
t_{M^2}(W)=\EE\left[\prod_{i=1}^{5}\rho(x_{Z_i})^2\right].
\]
Now let $\FF$ be the set $\{f_I:I\subseteq [5]\}$, 
where $f_I(x_I)=\rho(x_I)^2$ if $I=Z_i$ and $f_I=1$ otherwise.
Then, since $Z_1, \dots, Z_5$ constitute half the triples in $\binom{A}{3}$, $\tilde{f_I}(x_I)=\rho(x_I)$ if $|I|=3$ and $\tilde{f_I}=1$ otherwise.
Moreover,
\[
t_{M^2}(W)=\EE\left[\prod_{i=1}^{5}\rho(x_{Z_i})^2\right]
=\EE\left[\prod_{I\subseteq [5]} f_{\phi(I)}(x_I)\right]
\]
for each $\phi\in S_5$.
Hence, H\"{o}lder's inequality~\eqref{eq:Holder} implies that
\begin{align*}
t_{F_{5,3}}(W) & =\EE\left[\prod_{T\in \binom{5}{3}}\rho(x_{T})\right] = \EE\left[\prod_{I\subseteq [5]}\tilde{f_I}(x_I)\right]\\
& \leq \prod_{\phi\in S_5}\EE\left[\prod_{I\subseteq [5]}f_{\phi(I)}(x_I)\right]^{1/5!}
= \prod_{\phi\in S_5}
\EE\left[\prod_{i=1}^{5}\rho(x_{Z_i})^2\right]^{1/5!}=t_{M^2}(W),
\end{align*}
as required.
\end{proof}

\section{Weighted homomorphism densities} \label{sec:frac}

Let $\alpha=(\alpha_v)_{v\in B}$ be a vector indexed by $v\in B$ with non-negative coordinates.
Define the \emph{$\alpha$-weighted homomorphism density} of $H$ by
\begin{align*}
	t_H^{\alpha}(W)=\EE\left[\prod_{v\in B}\rho(x_{N(v)})^{\alpha_v}\right],
\end{align*}
where the expectation here and below, unless otherwise indicated, is over $x_A$. In particular, if $\alpha_v=p$ for each $v\in B$, then $t_H^{\alpha}(W)=t_{H_A^p}(W)$. We say that the weight vector $\alpha=(\alpha_v)_{v\in B}$ is \emph{symmetric} if
$\alpha_u=\alpha_v$ whenever $\deg(u)=\deg(v)$, i.e., it assigns the same weight to vertices with the same degree.
When $\alpha$ is symmetric, we simply write $\alpha = (\alpha_k)_{k=1}^r$, where $r$ is the maximum degree of $B$ and $\alpha_k:=\alpha_{v}$ for all $v$ with $\deg(v)=k$.

For $r \leq m$, define the \emph{$(m,r)$-downset graph} to be the $(A,\HH)$-incidence graph with $A = [m]$ and $\HH=\binom{A}{\leq r}$.
The main result of this section states that, for certain symmetric integer weight vectors $\alpha$, a `weighted' version of Sidorenko's conjecture holds for the $(m,r)$-downset graph. This can also be interpreted as saying that Sidorenko's conjecture holds for certain blow-ups of the $(m,r)$-downset graph, where vertices in $\binom{[m]}{\leq r}$ of different degrees may be blown up by different amounts.

\begin{theorem}\label{thm:downset}
Suppose $r \leq m$ and let $H$ be the $(m,r)$-downset graph and $\alpha=(\alpha_k)_{k=1}^r$ be a symmetric integer weight vector such that $\binom{m-k}{r-k}$ divides $\alpha_k$ for each $1\leq k\leq r$ and $\alpha_r>0$. Then, for every graphon~$W$,
\begin{align}\label{eq:fracSido}
	t_{H}^{\alpha}(W)\geq t_{K_2}(W)^{e_\alpha(H)},
\end{align}
where $e_\alpha(H):=\sum_{v\in B}\alpha_{v}\deg(v)$.
\end{theorem}

The proof has three steps. First, we construct a weakly norming $(r+1)$-graph $\HH_{\alpha}$
and a measurable function $W_\alpha:[0,1]^{r+1}\rightarrow[0,1]$ such that
\[
t_H^{\alpha}(W)= t_{\HH_\alpha}(W_\alpha).
\]
Then, by using the fact that $\HH_\alpha$ is weakly norming, we obtain
\[
t_{\HH_\alpha}(W_\alpha)\geq t_{\GG_
\alpha}(W_\alpha)^{e(\HH_\alpha)/e(\GG_\alpha)},
\]
where $\GG_{\alpha}$ is a subgraph of $\HH_\alpha$ with `simpler' structure.
We then conclude by showing that
\[
t_{\GG_\alpha}(W_\alpha)\geq t_{K_2}(W)^{q_{\alpha,H}},
\]
where $q_{\alpha,H}$ is the `correct' exponent to yield \eqref{eq:fracSido}.

\medskip

Throughout this section, let $\beta_k:=\alpha_k/\binom{m-k}{r-k}$ for $k=1,2,\cdots,r$ and $\beta:=\beta_1\beta_2\cdots\beta_r$.
The divisibility condition in Theorem~\ref{thm:downset} ensures that each $\beta_k$ is an integer.
For an $r$-set~$F$,
let $U_0(F):=F$ and, for $1 \leq k\leq r$, let $U_k(F)$ be the union of $\beta_k$ disjoint copies of $\binom{F}{k}$. 
We say that $(u_0,u_1,u_2,\cdots,u_r)$, $u_i\in U_i(F)$, is an \emph{$F$-chain} if the corresponding subsets of $F$ form a chain under containment.

Let $V_0=[m]$ and, for $1 \leq k \leq r$, let $V_k$ be the disjoint union of all $U_k(F)$, $F\in \binom{[m]}{r}$.
In other words, each $v\in V_k$ corresponds to the $j$-th copy of the $k$-set $F'$ in $U_k(F)$
for some $1 \leq j \leq \beta_k$, $F'\in \binom{F}{k}$, and $F\in \binom{[m]}{r}$.
Since $V_k$ consists of $\beta_k$ copies of each $F'\in\binom{[m]}{k}$ for each $F \in\binom{[m]}{r}$ with $F' \subseteq F$, we have $|V_k|=\binom{m}{k}\binom{m-k}{r-k}\beta_k=\binom{m}{r}\binom{r}{k}\beta_k$.
For each $v\in V_k$, let $c_k(v)$ denote the set $u\in U_k(F)$ that corresponds to $v$.
We then define $\HH_\alpha$ to be the $(r+1)$-partite $(r+1)$-graph on $V = V_0\cup V_1\cup\cdots\cup V_r$ where $(v_0,v_1,\cdots,v_r)$, $v_i\in V_i$, is an edge if and only if the corresponding subsets $c_0(v_0),c_1(v_1),\cdots,c_r(v_r)$ form an $F$-chain for some $F\in \binom{[m]}{r}$.

To help understand this definition, let us also define the subgraph $\GG_\alpha$ which will appear in Lemmas~\ref{lem:norming} and \ref{lem:G_a}. Write $U_i'=U_i([r])$ and let $\GG_\alpha$ be the $(r+1)$-partite $(r+1)$-graph on $U_0'\cup U_1'\cup\cdots\cup U_{r}'$
where $(u_0,u_1,\cdots,u_r)$, $u_i \in U_i'$, is an edge if and only if 
it is an $[r]$-chain.
Clearly,~$\GG_\alpha$ is isomorphic to each subgraph of $\HH_\alpha$ induced on 
$U_0(F)\cup U_1(F)\cup\cdots\cup U_r(F)$ for an $r$-set $F$.
Moreover, these induced subgraphs isomorphic to $\GG_\alpha$ are edge disjoint, so $e(\HH_\alpha)=\binom{m}{r}e(\GG_\alpha)$.

We make two further remarks about the definition of $\HH_\alpha$. First, note that $V_0$ has a different status to the rest of the $V_k$ in that, for each $F \in \binom{[m]}{r}$, $U_0(F)$ is identified with the subset $F$ of $V_0 = [m]$, while, for all other $k$, the $U_k(F)$ are all disjoint subsets of $V_k$. Thus, the $\binom{m}{r}$ edge-disjoint copies of $\GG_\alpha$ that decompose $\HH_\alpha$ are close to being vertex disjoint in the sense that they can only intersect in $V_0$. Second, if any $\beta_k = 0$ for $k<r$, we simply ignore the $k$-th coordinate. This will reduce the uniformity of the hypergraph, but all of our arguments still go through in this case. For convenience of notation, we assume in what follows that $\beta_k \neq 0$ for all $k$.

Now let $W_\alpha:[0,1]^{r+1}\rightarrow [0,1]$ be the function
\[
W_{\alpha}(x,z_1,z_2,\cdots,z_r)=\prod_{k=1}^r W(x,z_k)^{q_k},
\]
where 
\[
q_k=\frac{\beta_k}{\beta(k-1)!(r-k)!}.
\]
Our first step is contained in the next lemma.

\begin{lemma}\label{lem:star_hypergraph}
Suppose $r \leq m$ and let $H$ be the $(m,r)$-downset graph and $\alpha=(\alpha_k)_{k=1}^{r}$ be a symmetric integer weight vector such that $\binom{m-k}{r-k}$ divides $\alpha_k$ for each $1\leq k\leq r$. Then, for every graphon $W$,
\[
t_H^{\alpha}(W)= t_{\HH_\alpha}(W_\alpha).
\]
\end{lemma}

\begin{proof}
For $F\in \binom{[m]}{r}$, let $g(x_F):=\prod_{F'\subseteq F}\rho(x_{F'})^{\beta_{|F'|}}$.
Then, since any $k$-set is contained in exactly $\binom{m-k}{r-k}$ $r$-sets, we may rewrite $t_H^{\alpha}(W)$ as
\begin{align}\label{eq:distribute}
t_H^{\alpha}(W)=\EE\left[\prod_{k=1}^{r}\prod_{F\in \binom{[m]}{k}}\rho(x_{F})^{\alpha_k}\right]
=\EE\left[\prod_{F\in \binom{[m]}{r}}g(x_F) \right].
\end{align}
For each $k$-set $F'$, note that 
\[
\rho(x_{F'})^{\beta_k}
=\EE_{y_1,\cdots,y_{\beta_k}}\prod_{i\in F'}\prod_{j=1}^{\beta_k}W(x_i,y_j).
\]
Now relabel all the $y_j$, $j=1,2,\cdots,\beta_k$, for each $F'\in\binom{F}{k}$, 
by mutually independent uniform random variables $z_{v}$, where $v$ ranges over those $v\in U_k(F)$.
Let $U(F):=U_1(F)\cup U_2(F)\cup\cdots\cup U_r(F)$
and write $i\sim v$ if $i$ is contained in the $k$-subset of $F$ that corresponds to $v$. Then, for each $F\in \binom{[m]}{r}$,
\[
g(x_F)=\prod_{F'\subseteq F}\rho(x_{F'})^{\beta_{|F'|}} 
= \EE_{z_{v}: v\in U(F)}\left[\prod_{v\in U(F)}\prod_{i\sim v}W(x_i,z_{v})\right],
\]
since the $z_v$, $v\in U(F)$, are mutually independent.
We may repeat this step for each $r$-set $F$ while assigning mutually independent random variables $z_v$ for all $v\in \bigcup_{F\in \binom{[m]}{r}}U(F)$.
Then, by \eqref{eq:distribute},
\begin{align}\label{eq:huge_product}
t_H^{\alpha}(W)
&=\EE\left[\prod_{F\in \binom{[m]}{r}}g(x_F) \right]
= \EE_{x_{[m]}}\left[\prod_{F\in\binom{[m]}{r}}\EE_{z_{v}: v\in U(F)}\left[\prod_{v\in U(F)}\prod_{i\sim v}W(x_i,z_{v})\right]\right] \nonumber \\
&=\EE\left[\prod_{F\in\binom{[m]}{r}}\prod_{v\in U(F)}\prod_{i\sim v}W(x_i,z_{v})\right],
\end{align}
where the last equality follows from mutual independence. 
We shall verify that the right-hand side equals $t_{\HH_\alpha}(W_\alpha)$ by comparing the exponents of the $W(x_i,z_v)$ in \eqref{eq:huge_product} with those in 
\begin{align}\label{eq:hyper_product}
t_{\HH_\alpha}(W_\alpha)
=\EE\left[\prod_{(i,v_1,\cdots,v_r)\in E(\HH_{\alpha})}\prod_{k=1}^r W(x_i,z_{v_k})^{q_k}\right].
\end{align}
For each pair consisting of $i\in [m]$ and $v\in U_k(F)$ with $i\sim v$,
the term $W(x_i,z_v)$ appears exactly once in the product in~\eqref{eq:huge_product}.
On the other hand, there exist $\beta(k-1)!(r-k)!/\beta_k$ $F$-chains 
that contain $\{i\}$ and $v$.
Thus, there are $\beta(k-1)!(r-k)!/\beta_k$ hyperedges in $\HH_{\alpha}$ containing
$i$ and $v$.
Hence, by expanding the product in~\eqref{eq:hyper_product}, each $W(x_i,z_{v})$ receives
the exponent 
\[
\frac{\beta (k-1)!(r-k)! q_k}{\beta_k}=1,
\]
where we used the definition of $q_k$.
\end{proof}

The next lemma gives a lower bound for the homomorphism density $t_{\HH_\alpha}(W_\alpha)$ in terms of $t_{\GG_\alpha}(W_\alpha)$. This follows in a straightforward manner from a result of the authors~\cite{CL16}, so we have consigned the proof, and a broader discussion of weakly norming hypergraphs, to an appendix.

\begin{lemma}\label{lem:norming}
The $(r+1)$-graph $\HH_\alpha$ is weakly norming. 
In particular, for every $(r+1)$-graphon~$W_\alpha$, 
\[t_{\HH_\alpha}(W_\alpha)\geq 
t_{\GG_\alpha}(W_\alpha)^{\binom{m}{r}}.\]
\end{lemma}

The following lemma is the final ingredient we need to prove Theorem~\ref{thm:downset}.

\begin{lemma}\label{lem:G_a}
If $\binom{m-k}{r-k}$ divides $\alpha_k$ for each $1\leq k\leq r$ and $\alpha_r > 0$, then, for every graphon $W$,
\[
t_{\GG_\alpha}(W_\alpha)\geq t_{K_2}(W)^{q_{\alpha,H}},
\]
where $q_{\alpha,H}=e_\alpha(H)/\binom{m}{r}$.
\end{lemma}

\begin{proof}
Let $U'=U_1'\cup U_2'\cup\cdots\cup U_r'$, where the $U'_i$ are as in the definition of $\GG_\alpha$.
Then, following the proof of Lemma~\ref{lem:star_hypergraph}, we have
\[
t_{\GG_\alpha}(W_\alpha)
=\EE\left[\prod_{(i,u_1,\cdots,u_r)\in E(\GG_{\alpha})}\prod_{k=1}^r W(x_i,z_{u_k})^{q_k}\right]
=\EE\left[\prod_{i\sim u, i\in [r],u\in U'}W(x_i,z_{u})\right],
\]
where $i\sim u$ means that $i$ is contained in the subset of $[r]$ corresponding to $u$.
Hence, we may write
\begin{align}\label{eq:claim}
t_{\GG_\alpha}(W_\alpha)=\EE\left[\EE_{z_u:u\in U'}\left[\prod_{i\sim u, i\in [r],u\in U'}W(x_i,z_{u})\right]\right]=\EE\left[\rho(x_{[r]})^{\beta_r}\prod_{F\subsetneq [r]}\rho(x_{F})^{\beta_{|F|}}\right]=t_J^\beta(W),
\end{align}
where $J$ is the $(r,r)$-downset graph and $\beta=(\beta_i)_{i=1}^r$ is a symmetric integer weight vector with $\beta_r=\alpha_r\geq 1$.
Since $t_J^\beta(W)$ can be interpreted as the homomorphism density in $W$ of a bipartite graph with at least one vertex complete to the other side, the result of Conlon, Fox and Sudakov~\cite{CFS10} implies that
\[
t_J^\beta(W)\geq t_{K_2}(W)^{e_\beta(J)},
\]
where $e_\beta(J)=\sum_{k=1}^r\binom{r}{k}k\beta_k$.
By the elementary identity $\binom{m}{r}\binom{r}{k}=\binom{m}{k}\binom{m-k}{r-k}$,
\[
e_\beta(J)=\frac{1}{\binom{m}{r}}\sum_{k=1}^r\binom{m}{k}\binom{m-k}{r-k}k\beta_k
=\frac{1}{\binom{m}{r}}\sum_{k=1}^r\binom{m}{k}k\alpha_k
=\frac{e_\alpha(H)}{\binom{m}{r}},
\]
as desired.
\end{proof}

\begin{proof}[Proof of Theorem~\ref{thm:downset}]
By Lemmas~\ref{lem:star_hypergraph}, \ref{lem:norming}, and \ref{lem:G_a}, we obtain
\[
t_H^{\alpha}(W)= t_{\HH_{\alpha}}(W_\alpha)
\geq t_{\GG_{\alpha}}(W_\alpha)^{\binom{m}{r}}\geq t_{K_2}(W)^{e_\alpha(H)},
\]
as required.
\end{proof}

\section{The H\"{o}lder trick}\label{sec:Holder}

The argument in Section~\ref{sec:example} that $M^2$ satisfies Sidorenko's conjecture for $M = K_{5,5}\setminus C_{10}$ had two ingredients, the fact that the $(5,3)$-incidence graph $F_{5,3}$ is weakly norming and, therefore, satisfies the conjecture, and the inequality $t_{M^2}(W) \geq t_{F_{5,3}}(W)$ for all graphons $W$. The main result of the previous section may be seen as a generalisation of the fact that $F_{5,3}$ satisfies the conjecture. To complete the proof of Theorem~\ref{thm:main}, we now generalise the inequality. This will again be a simple consequence of H\"older's inequality. 

\begin{theorem}\label{thm:Holder}
Let $H$ be a bipartite graph with bipartition $A\cup B$ and $\max_{b\in B}\deg(b)=r$ and, for each $1 \leq k \leq r$, let $\alpha_k:=d_k/\binom{|A|}{k}$, where $d_k$ is the number of vertices in $B$ of degree $k$.
Then, for every graphon $W$,
\begin{align*}
 t_H(W)\geq t_{J}^{\alpha}(W),
\end{align*}
where $J$ is the $(|A|,r)$-downset graph and $\alpha=(\alpha_k)_{k=1}^r$. Moreover, if $\deg(b)=r$ for all $b\in B$, then, for every graphon $W$, 
\[t_H(W) \geq t_{F}^\alpha (W),\]
where $F$ is the $(|A|, r)$-incidence graph and $\alpha=(\alpha_k)_{k=1}^r$ is the symmetric weight vector with $\alpha_r=|B|/\binom{|A|}{r}$ and $\alpha_k=0$ for $k\neq r$.
\end{theorem}

Theorem~\ref{thm:main} is an immediate consequence of this result and Theorem~\ref{thm:downset}. 
Recall the statement, that Sidorenko's conjecture holds for any bipartite graph $H$ with bipartition $A \cup B$ and $\max_{b \in B} \deg(b) = r$ such that 
the number of vertices of degree $k$ is divisible by $\binom{|A|}{r}\binom{r}{k}=\binom{|A|}{k}\binom{|A|-k}{r-k}$ for each $1 \leq k \leq r$.

\begin{proof}[Proof of Theorem~\ref{thm:main}]
As in Theorem~\ref{thm:Holder}, for each $1 \leq k \leq r$,
let $d_k$ be the number of vertices in $B$ of degree $k$
and $\alpha_k:=d_k/\binom{|A|}{k}$. Then, by Theorem~\ref{thm:Holder},
\[t_H(W)\geq  t_{J}^{\alpha}(W),\]
where $J$ is the $(|A|,r)$-downset graph.
By the divisibility assumption, $\alpha=(\alpha_k)_{k=1}^r$ is a nonnegative integer vector such that $\binom{|A|-k}{r-k}$ divides $\alpha_k$ for each $1\leq k\leq r$ and $\alpha_r>0$.
Hence, we may apply Theorem~\ref{thm:downset} to conclude that
\begin{align*}
t_J^\alpha(W)\geq t_{K_2}(W)^{e_{\alpha}(J)}.
\end{align*}
Therefore, since
\[e_{\alpha}(J) = \sum_{k=1}^r \binom{|A|}{k} \alpha_k k = \sum_{k=1}^r d_k k = e(H),\]
Sidorenko's conjecture holds for $H$.
\end{proof}

Theorem~\ref{thm:reg} is a similar consequence of the second part of Theorem~\ref{thm:Holder} with one slight twist.

\begin{proof}[Proof of Theorem~\ref{thm:reg}]
With the slight abuse of notation $\alpha = \alpha_r = |B|/\binom{|A|}{r}$, we have
\begin{align*}
	t_{F}^{\alpha}(W)=\EE\left[\prod_{I\in \binom{A}{r}}\rho(x_{I})^{\alpha}\right]
	\geq \EE\left[\prod_{I\in \binom{A}{r}}\rho(x_{I})\right]^{\alpha} = t_{F}(W)^\alpha,
\end{align*}
where the inequality follows from the convexity of $f(x)=x^{\alpha}$ when $\alpha\geq 1$.
Therefore,
\begin{align*}
    t_H(W) \geq t_{F}^{\alpha}(W) \geq t_{F}(W)^{\alpha} \geq t_{K_2}(W)^{\alpha r\binom{|A|}{r}} = t_{K_2}(W)^{e(H)},
\end{align*}
where the first inequality is the second part of Theorem~\ref{thm:reg}, the third inequality follows from the fact that the $(|A|,r)$-incidence graph $F$ is weakly norming~\cite{CL16} and, therefore, satisfies Sidorenko's conjecture, and the final equality comes from the straightforward computation $\alpha r\binom{|A|}{r}=r|B|=e(H)$.
\end{proof}

We now return to the proof of Theorem~\ref{thm:Holder}.

\begin{proof}[Proof of Theorem~\ref{thm:Holder}]
Let $m:=|A|$ and identify the set $A$ with $[m]$.
Let $\FF=\{f_I:I\subseteq [m]\}$ be the family of functions 
given by $f_I(x_I)=\rho(x_I)^{c_I}$, where $c_I$ is the number of vertices $b\in B$ such that $N(b)=I$.
If $I$ and $I'$ are subsets of $[m]$ of size $k$, 
then there are exactly $k!(m-k)!$ permutations $\phi$
that map $I$ onto $I'$
and, for every such $\phi$, $f_{\phi(I)}(x_I)=\rho(x_I)^{c_{I'}}$.
Thus, for each $I$ of size $k$,
\[
\tilde{f}_I(x_I)=\left(\prod_{\phi\in S_m}f_{\phi(I)}(x_I)\right)^{1/m!}
=\rho(x_I)^{\sum_{|K|=k}c_K/\binom{m}{k}}=\rho(x_I)^{\alpha_k}.
\]
On the other hand, since each $\phi\in S_m$ is just a relabeling of $A=[m]$,
\[
t_H(W) = \EE\left[\prod_{I\subseteq [m]}f_{\phi(I)}(x_I)\right].
\]
Therefore, H\"{o}lder's inequality in the form~\eqref{eq:Holder} gives
\[
t_{J}^{\alpha}(W)=
\EE\left[\prod_{I\subseteq [m]}\rho(x_I)^{\alpha_{|I|}}\right]=
\EE\left[\prod_{I\subseteq [m]}\tilde{f_I}(x_I)\right]\leq \prod_{\phi\in S_m}\EE\left[\prod_{I\subseteq [m]}f_{\phi(I)}(x_I)\right]^{1/m!}=t_H(W),
\]
as desired. Moreover, if $\deg(b)=r$ for all $b\in B$, then $\alpha_r=|B|/\binom{m}{r}$ and $\alpha_k=0$ for $k\neq r$. Thus,
\begin{align*}
    t_H(W) \geq \EE\left[\prod_{I\in\binom{[m]}{r}}\rho(x_I)^{\alpha_{r}}\right] = t_{F}^{\alpha}(W),
\end{align*}
which proves the second part of the statement.
\end{proof}

\section{Concluding remarks}

As a closing remark, we note another corollary of  Theorem~\ref{thm:main} which we believe to be of independent interest. Put briefly, it says that a certain local version of Sidorenko's conjecture holds for any bipartite graph $H$.

\begin{corollary}
For every bipartite graph $H$ with bipartition $A\cup B$ and every $\varepsilon, q>0$, there is $n_0=n_0(\varepsilon,H,q)$ such that every $n$-vertex graph $G$ with $n\geq n_0$ and $q=t_{K_2}(G)$ has an $|A|$-tuple $x_A$ of distinct vertices such that $t_H(G;x_A)\geq (1-\varepsilon)q^{e(H)}$. 
\end{corollary}

\begin{proof}
Let $a=|A|$ and $n=|V(G)|$. Let $K$ be the set of non-degenerate $a$-tuples in $V(G)^{A}$, noting that $|V(G)^{A}\setminus K| \leq an^{a-1}$. Then
\begin{align*}
    \EE_{x_A} t_H(G;x_A)^k = \frac{1}{n^{a}} \left(\sum_{x_A\in K} t_H(G;x_A)^k+ \sum_{x_A\notin K} t_H(G;x_A)^k\right)\leq \frac{1}{n^a}\sum_{x_A\in K} t_H(G;x_A)^k +\frac{a}{n}.
\end{align*}
Therefore, by letting $k=a!$, Corollary~\ref{cor:main} implies that
\begin{align*}
    \EE_{x_A\in K} t_H(G;x_A)^k
    &\geq \frac{n^a}{n(n-1)\cdots (n-a+1)}\left(\EE_{x_A} t_H(G;x_A)^k  - \frac{a}{n}\right)
    \\ 
    &\geq (1-\varepsilon) t_{K_2}(G)^{k\cdot e(H)},
\end{align*}
provided $n$ is sufficiently large in terms of $\epsilon$, $H$ and $q$. There must then be $x_A\in K$ with $t_H(G;x_A)\geq (1-\varepsilon)t_{K_2}(G)^{e(H)}$, since otherwise we would have $\EE_{x_A\in K} t_H(G;x_A)^k < (1 - \varepsilon)^k t_{K_2}(G)^{k\cdot e(H)}$, contradicting the inequality above.
\end{proof}

%%% AUTHOR: optional acknowledgments here
\section*{Acknowledgments} %%  you may comment this out if no Ackno
We are greatly indebted to Yufei Zhao and also to Leonardo Nagami Coregliano and Sasha Razborov for spotting a substantial error in an earlier version of this paper. The upshot of the resulting changes is the divisibility condition in Theorem~\ref{thm:main}, which was not present in the previous version. We are also grateful to the anonymous referees for their detailed reviews.

%%% AUTHOR:
%%% Bibliography goes here. Note that the arXiv cannot process bibtex
%%% or biber bibliographies.  Example of acceptable bibliograpy format:
\bibliographystyle{amsplain}
\bibliography{references}

%\newpage %% AUTHOR: please comment out this line.  It serves only
%%   to demonstrate both types of header line in daj-template.pdf

%\section{Expansion estimates}

% More of the body of your paper goes here~\cite{bergelson-johnson-moreira}.

%%% AUTHOR: optional appendix here
\appendix %% you may comment this out if no Appendix
\section*{Appendix}
\section{Weakly norming hypergraphs} \label{sec:norm}
Following Hatami~\cite{H10}, we say that a hypergraph $H$ is weakly norming if the functional $\|\cdot\|_{r(H)}$ defined by
\[\|W\|_{r(H)} = t_H(|W|)^{1/e(H)}\]
is a norm on the space of bounded symmetric measurable functions. 
As shown by Hatami~\cite{H10}, any weakly norming hypergraph $H$ has the property that for any $H' \subseteq H$ and any graphon $W$, 
\[\|W\|_{r(H)} \geq \|W\|_{r(H')}\]
or, in the language of homomorphism densities,
\[t_H(W) \geq t_{H'}(W)^{e(H)/e(H')}.\]
In particular, this implies that any weakly norming hypergraph satisfies Sidorenko's conjecture.

The main result of interest to us here is a result of the authors~\cite[Theorem 5.1]{CL16} saying that a certain class of hypergraphs, which we call reflection hypergraphs, are weakly norming. To define this class, suppose that $W$ is a finite reflection group, $T$ is a set of simple reflections in $W$, and $T_1, \cdots, T_r$ are subsets of $T$. Then the \emph{$(T_1, \cdots, T_r; T, W)$-reflection hypergraph} is the $r$-partite $r$-graph whose parts are the cosets of the subgroup $W_k$ generated by $T_k$ for each $k = 1, \cdots, r$, with an edge for every $r$-tuple of the form $(wW_1, \cdots, wW_r)$ with $w \in W$. An $r$-graph is then said to be a \emph{reflection hypergraph} if it is isomorphic to the $(T_1, \cdots, T_r; T, W)$-reflection hypergraph for some choice of parameters.

\begin{theorem}
Reflection hypergraphs are weakly norming.
\end{theorem}

In order to prove Lemma~\ref{lem:norming}, it therefore suffices to show that $\HH_\alpha$ is a reflection hypergraph.
If $\beta_k=1$ for each $k$, then 
it is not hard to construct a reflection hypergraph
that is isomorphic to~$\GG_\alpha$.
Let $s_{ij}$ be the permutation in $S_r$ that swaps $i$ and $j$.
It is well-known that $T:=\{s_{i(i+1)}:i=1,2,\cdots,r-1\}$ is a set of simple reflections.
We claim that the $(T_0,T_1,T_2,\cdots,T_r;T,S_r)$-reflection hypergraph with the choice $T_0=T_1$, $T_i=T\setminus \{s_{i(i+1)}\}$, $1\leq i<r$,
$T_r=T$,
is isomorphic to $\GG_\alpha$
when $\beta_1=\beta_2=\cdots=\beta_r=1$.
To see this,
observe first that each $T_k$, $1\leq k<r$, generates the subgroup
\[
W_k:=\{\sigma\in S_r:\sigma(j)\leq k\text{ for each }j\leq k\},
\]
which is isomorphic to $S_k\times S_{r-k}$.
Thus, each coset $wW_k$ is the set of permutations in $S_r$ that map $[k]$ onto some $k$-subset $F\in\binom{[r]}{k}$,
which allows us to identify each coset with the $k$-set $F$.
For example, the chain
$\{1\} = \{1\}\subset\{1,2\}\subset\cdots\subset[r-1]\subset[r]$, which is an edge in $\GG_\alpha$, corresponds to the $(r+1)$-tuple $(W_1,W_1,\cdots,W_{r-1},S_r)$. 
Similarly, every edge in $\GG_\alpha$, identified by a chain $\{i\}=F_1\subset\cdots\subset F_{r-1}\subset[r]$, where $|F_k|=k$, corresponds to an $(r+1)$-tuple $(wW_1,wW_1,\cdots,wW_{r-1},S_r)$ of cosets, where $w\in S_r$ is a permutation that maps $1$ to $i$.
This proves the claim.

Recall now that $\HH_\alpha$ consists of $\binom{m}{r}$ edge-disjoint $(r+1)$-graphs
$\HH_\alpha[U_0(F) \cup U(F)]$, $F\in \binom{[m]}{r}$,
each of which is isomorphic to~$\GG_\alpha$.
To realise $\HH_\alpha$ with $\beta_i=1$ for all $i=1,2,\cdots,r$ as a reflection hypergraph, 
one may add new reflections $s_{j(j+1)}$, for $j=r,\cdots,m-1$, to $T$,
and amend the $T_k$, $0 \leq k \leq r$, to generate $\binom{m}{r}$ copies of each coset $wW_k$, $1\leq k\leq r$, in the $(T_0,T_1,\cdots,T_r;T,S_r)$-reflection hypergraph.
More explicitly,
$\HH_\alpha$ is isomorphic to the $(T_0',T_1',T_2',\cdots,T_r';T',S_m)$-reflection hypergraph,
where $T'=T\cup\{s_{j(j+1)}:r\leq j<m\}$,
$T_0'=T_0\cup\{s_{j(j+1)}:r\leq j<m\}$
and 
$T_k'=T_k\cup\{s_{j(j+1)}:r+1\leq j<m\}$
for $0<k\leq r$.
To see this,
note first that the cosets of the subgroup generated by $T_0'$ and $T_r'$ 
correspond to singletons and $r$-subsets in $[m]$, respectively.
For $1\leq k<r$, $T_k'=T'\setminus\{s_{k(k+1)},s_{r(r+1)}\}$ generates the subgroup
\[
W_k':=\{\sigma\in S_m:\sigma(i)\leq k \text{ if }1\leq i\leq k,~ k<\sigma(j)\leq r \text{ if }k<j\leq r \},
\]
which is isomorphic to $S_k\times S_{r-k}\times S_{m-r}$.
Thus, each coset of $W_k'$ is identified with a disjoint pair $(F',F'')$ consisting of a $k$-set $F'$ and an $(r-k)$-set $F''$,
which corresponds to an element in $U_k(F'\cup F'')$.
For instance, $W_k'$ corresponds to the copy of $[k]$ in $U_k([r])$.
Then it is easy to check that each edge in the $(T_0',T_1',T_2',\cdots,T_r';T',S_m)$-reflection hypergraph corresponds to an $F$-chain for some $r$-set $F$.

We also remark that these constructions generalise to the case $\beta_k\in\{0,1\}$ for $k=1,2,\cdots,r$.
If some $\beta_k=0$, one may simply delete the corresponding $T_k$ or $T_k'$ when constructing $\GG_\alpha$ and $\HH_\alpha$, respectively.
By doing so, reflection $t$-graphs, where $t$ is the number of positive $\beta_k$, are obtained.

\medskip

Given the union of $r$ disjoint sets $A_1\cup A_2\cup\cdots\cup A_r$,
let $\overbar{A}_k=\cup_{i=1}^{r}A_i\setminus A_k$. Moreover, if $\HH$ is an $r$-partite $r$-graph with $r$-partition $A_1\cup A_2\cup\cdots\cup A_r$, the blow-up 
$\HH^p_{\overbar{A}_k}$ is obtained by taking $p$ vertex-disjoint copies of $\HH$ and gluing all copies of $\overbar{A}_k$ along corresponding vertices. The following lemma then says that blow-ups of reflection hypergraphs are also reflection hypergraphs.

\begin{lemma}
Let $\HH$ be an $r$-partite $r$-graph on $A_1\cup A_2\cup\cdots\cup A_r$.
If $\HH$ is a reflection hypergraph and $k$ and $p$ are positive integers with $1 \leq k \leq r$, then the blow-up $\HH^p_{\overbar{A}_k}$ is also a reflection hypergraph.
\end{lemma}

\begin{proof}
Suppose that $\HH$ is isomorphic to the $(T_1,T_2,\cdots,T_r;T,W)$-reflection hypergraph, where $W$ is a finite reflection group, $T$ is a set of simple reflections in $W$ and $T_i\subseteq T$, $i=1,2,\cdots,r$.
Let $P=\{\sigma_{12},\sigma_{23},\cdots,\sigma_{(p-1)p}\}$ be a set of $p-1$ new reflections such that each element of $P$ commutes with all the elements in $W$ and $P$ generates a group which is isomorphic to $S_p$,
where $\sigma_{ij}$ corresponds to the permutation swapping $i$ and $j$ in $[p]$.
Let $T':=T\cup P$, $T_i':=T_i\cup P$ for $i\neq k$, $T_k':=T_k\cup P\setminus\{\sigma_{12}\}$, and let $W'$ be
the new reflection group generated by $T'$, i.e., $W'\cong W\times S_{p}$.

Let $\HH'$ be the $(T_1',\cdots,T_r';T',W')$-reflection hypergraph.
We claim that $\HH'$ is isomorphic to $\HH_{\overbar{A}_k}^p$.
Let $A_i':=\{wW_i':w\in W\}$. Then $\HH'$ is an $r$-graph on the new $r$-partition $A_1'\cup A_2'\cup \cdots\cup A_r'$.
Moreover, $A_i'=A_i \times S_p$ for all $i\neq k$ since $T_i'$ contains all the new reflections in $T$,
whereas $A_k'=\cup_{j=1}^{p} (A_k \times \sigma_{1j} S_{[2,p]})$.
Let $\phi:V(\HH')\rightarrow V(\HH)$
be the map defined by 
\[
\phi(wW_i \times S_p):= wW_i \text{ for each } i\neq k ~\text{ and }~
\phi(wW_k \times \sigma_{1j} S_{[2,p]}):= wW_k \text{ for each } 1 \leq j \leq p.
\]
One may check that $\phi$ is a surjective hypergraph homomorphism from $\HH'$ to $\HH$.
Moreover, it is injective on $\cup_{i\neq k} A_i'$, $p$-to-1 on $A_k'$, and maps every non-edge in $\HH'$ to a non-edge in $\HH$.
Therefore, $\HH'$ is obtained by gluing $p$ vertex-disjoint copies of the $(T_1,\cdots,T_r;T,W)$-reflection hypergraph
along the vertices that are not cosets of $W_k$
and, hence, it is isomorphic to $\HH_{\overbar{A}_k}^p$.
\end{proof}

\begin{proof}[Proof of Lemma~\ref{lem:norming}]
Let $\HH$ be the hypergraph $\HH_\alpha$ with $\beta_k\in\{0,1\}$ for all $k$.
By applying the lemma above with this $\HH$,
one may blow-up each vertex in the vertex set $V_k$ of $\HH$ to $\beta_k$ copies.
Thus, by repeating this blowing-up process, we conclude that $\HH_{\alpha}$ is a reflection hypergraph for any given 
non-negative integers $\beta_1,\beta_2,\cdots,\beta_r$.
\end{proof}

%\begin{thebibliography}{99}
%\bibitem{bergelson-johnson-moreira}
%Vitaly Bergelson, John H. Johnson Jr., and Joel Moreira.
%\newblock New polynomial and multidimensional extensions of classical partition
%  results.
%\newblock 2015, arXiv:1501.02408.

%\bibitem{cilleruelo}
%Javier Cilleruelo.
%\newblock Combinatorial problems in finite fields and {S}idon sets.
%\newblock {\em Combinatorica}, 32(5):497--511, 2012.
%
%\end{thebibliography}
%% AUTHOR: You can generate such a bibliography from a .bib file by 
%% running pdflatex/bibtex/pdflatex/pdflatex and then pasting the .bbl file
%% between \begin{thebibliography} and \end{bibliography}

%%% AUTHOR: Include a short description of each author following the
%%% structure below. Use the same short tags used previously.  
%%% Use \imageat{} and \imagedot{} instead of "@" and "." in
%%% email addresses-this replaces the symbols with graphics to avoid 
%%% e-mail address harvesting from the .pdf file
\begin{dajauthors}
\begin{authorinfo}[dconlon]
    David Conlon\\
    Department of Mathematics\\
    California Institute of Technology\\
    Pasadena, CA 91125\\
    USA\\
    dconlon\imageat{}caltech\imagedot{}edu
\end{authorinfo}
\begin{authorinfo}[jlee]
  Joonkyung Lee\\
%  IMSS Research Fellow\\
  Department of Mathematics\\
  University College London\\
  London WC1E 6BT\\ 
  United Kingdom\\
  joonkyung.lee\imageat{}ucl\imagedot{}ac\imagedot{}uk
\end{authorinfo}
\end{dajauthors}

\end{document}